\documentclass[12pt,reqno,oneside,english]{amsart}

\usepackage[left=2.75cm,right=2.75cm,top=3.25cm,bottom=3.25cm]{geometry}
\usepackage{amssymb}
\usepackage{graphicx}
\usepackage{amscd}
\usepackage[hidelinks]{hyperref}
\usepackage{color}
\usepackage{float}
\usepackage{graphics,amsmath,amssymb}
\usepackage{amsthm}
\usepackage{amsfonts}
\usepackage{latexsym}
\usepackage{epsf}
\usepackage{xifthen}
\usepackage{mathrsfs}
\usepackage{dsfont}
\usepackage{makecell}
\usepackage[FIGTOPCAP]{subfigure}
\usepackage{amsmath}
\allowdisplaybreaks[4]
\usepackage{listings}
\usepackage{etoolbox}
\usepackage{fancyhdr}
\usepackage{pdflscape}
\usepackage[title,toc,titletoc]{appendix}
\usepackage{enumitem}
\usepackage[noadjust]{cite}

\makeatletter

\numberwithin{equation}{section}

\theoremstyle{plain}
\newtheorem{theorem}{Theorem}[section]
\newtheorem*{theorem*}{Theorem}

\newtheorem{corollary}[theorem]{Corollary}
\newtheorem{lemma}[theorem]{Lemma}

\newtheorem{innercustomgeneric}{\customgenericname}
\providecommand{\customgenericname}{}
\newcommand{\newcustomtheorem}[2]{%
	\newenvironment{#1}[1]
	{%
		\renewcommand\customgenericname{#2}%
		\renewcommand\theinnercustomgeneric{##1}%
		\innercustomgeneric
	}
	{\endinnercustomgeneric}
}
\newcustomtheorem{ctheorem}{Theorem}
\newcustomtheorem{clemma}{Lemma}

\theoremstyle{definition}
\newtheorem{definition}[theorem]{Definition}

\newtheorem*{example*}{Example}
\newtheorem*{examples*}{Examples}
\newtheorem{remark}[theorem]{Remark}
\newtheorem*{remark*}{Remark}
\newtheorem*{remarks*}{Remarks}

\newtheoremstyle{named}{}{}{\itshape}{}{\bfseries}{.}{.5em}{#1\thmnote{ #3}}
\theoremstyle{named}

\author[S. Chern]{Shane Chern} 
\address[S. Chern]{Department of Mathematics and Statistics, Dalhousie University, Halifax, Nova Scotia, B3H 4R2, Canada} 
\email{chenxiaohang92@gmail.com}

\author[L. Jiu]{Lin Jiu} 
\address[L. Jiu]{Zu Chongzhi Center for Mathematics and Computational Sciences, Duke Kunshan University, Kunshan, Suzhou, Jiangsu Province, 215316, P.R. China} 
\email{lin.jiu@dukekunshan.edu.cn}
\author[S. Li]{Shuhan Li} 
\address[S. Li]{Class of 2024, Duke Kunshan University, Kunshan, Suzhou, Jiangsu Province, 215316, P.R. China} 
\email{shuhan.li371@dukekunshan.edu.cn}

\author[L. Wang]{Liuquan Wang} 
\address[L. Wang]{School of Mathematics and Statistics, Wuhan University, Wuhan, Hubei Province, 430072, P.R. China} 
\email{wanglq@whu.edu.cn; mathlqwang@163.com}

\DeclareMathOperator{\sech}{sech}

\makeatother

\usepackage{babel}

\newcommand{\qbinom}[2]{{#1\brack #2}}

\newcommand{\eval}{\mathsf{Eval}}

\begin{document}

\title[Leading coefficient in Hankel determinants]{Leading coefficient in the Hankel determinants related to binomial and $q$-binomial transforms}

\begin{abstract}
It is a standard result that the Hankel determinants for a sequence stay invariant after performing the binomial transform on this sequence. In this work, we extend the scenario to $q$-binomial transforms and study the behavior of the leading coefficient in such Hankel determinants. We also investigate the leading coefficient in the Hankel determinants for even-indexed Bernoulli polynomials with recourse to a curious binomial transform. In particular, the degrees of these Hankel determinants share the same nature as those in one of the $q$-binomial cases.
\end{abstract}



\sloppy

\maketitle

\section{Introduction}

Given a sequence $(a_{k})_{k\ge 0}$, its $n$-th \emph{Hankel determinant} is the determinant of the \emph{Hankel matrix} 
\begin{align*}
	H_{n}(a_{k}):=\det_{0\leq i,j,\leq n}(a_{i+j})=\det\begin{pmatrix}a_{0} & a_{1} & a_{2} & \cdots & a_{n}\\
		a_{1} & a_{2} & a_{3} & \cdots & a_{n+1}\\
		\vdots & \vdots & \vdots & \ddots & \vdots\\
		a_{n} & a_{n+1} & a_{n+2} & \cdots & a_{2n}
	\end{pmatrix}.
\end{align*}
The study of Hankel determinants has been extensively developed, mainly based on the close relationship to classical orthogonal polynomials and continued fractions; see, e.g., \cite[Chapter 2]{Ismail}. 

The choice of $(a_{k})_{k\ge 0}$ relies on various reasons, but it is usually of general interest to consider classical number-theoretic sequences. For example, the \emph{Bernoulli numbers} $B_{k}$ and \emph{Bernoulli polynomials} $B_{k}(x)$ are defined by their exponential generating functions: 
\begin{align*}
	\frac{t}{e^{t}-1}=\sum_{k=0}^{\infty}B_{k}\frac{t^{k}}{k!}\qquad\text{and}\qquad\frac{te^{xt}}{e^{t}-1}=\sum_{k=0}^{\infty}B_{k}(x)\frac{t^{k}}{k!},
\end{align*}
with a simple connection that
\begin{align*}
	B_{k}(x)=\sum_{\ell=0}^{k}\binom{k}{\ell}B_{k-\ell}x^{\ell}.
\end{align*}
Al-Salam and Carlitz \cite[p.~93, Eq.~(3.1)]{AlSalamCarlitz1} discovered that
\begin{align}\label{eq:Bk}
H_{n}(B_{k})=(-1)^{\binom{n+1}{2}}\prod_{j=1}^{n}\frac{(j!)^{6}}{(2j)!(2j+1)!}.
\end{align}
For the Bernoulli polynomials $B_k(x)$, it is surprising that their Hankel determinants stay invariant for any choice of the argument $x$:
\begin{align}\label{eq:Hankel-B-poly}
	H_{n}(B_{k}(x))=H_{n}(B_{k}).
\end{align}
In general, there is a classical result on the invariance of Hankel determinants under the binomial transform 
\begin{equation}
	a_{k}(x):=\sum_{\ell=0}^{k}\binom{k}{\ell}a_{k-\ell}x^{\ell} \label{eq:BinomialTransform}
\end{equation}
for a generic sequence $(a_{k})_{k\ge 0}$. See \cite[p.~419, Entry 445]{Muir} for a direct proof, or an alternative proof in \cite[p.~393, Eq.~(10)]{JiuShi} for the scenario wherein $a_{k}$ is understood as the $k$-th moment of a certain random variable. 

\begin{theorem}[Invariance of Hankel determinants under the binomial transform]
\label{prop:BinomialTransform}
For every $n\ge 0$, define $a_k(x)$ as in \eqref{eq:BinomialTransform}. Then,
\begin{equation}\label{eq:inv}
    H_{n}(a_{k}(x))=H_{n}(a_{k}).
\end{equation}
\end{theorem}

The story becomes much more intricate when considering non-consecutive terms in sequences of polynomials as in \eqref{eq:BinomialTransform}. In early work, Dilcher and Jiu \cite[p.~3, Theorem 1.1]{DilcherJiu} evaluated the Hankel determinants for odd-indexed Bernoulli polynomials. Namely,
\begin{align*}
	H_{n}\left(B_{2k+1}\left(\frac{1+x}{2}\right)\right)=(-1)^{\binom{n+1}{2}}\left(\frac{x}{2}\right)^{n+1}\prod_{j=1}^{n}\left(\frac{j^{4}(x^{2}-j^{2})}{4(2j+1)(2j-1)}\right)^{n+1-j}.
\end{align*}
Here, the argument is chosen as $(1+x)/2$ instead of $x$, to obtain the better factorization of these Hankel determinants. On the other hand, the even-indexed case remains open \cite[pp.~10--11]{DilcherJiu}. It is notable that in view of the first few evaluations as shown in Table~\ref{tab:Even}, these Hankel determinants neither factor nor exhibit any obvious pattern. In fact, there are merely some partial results on special choices of the argument $x$. For instance, Chen \cite[p.~390, Corollary 5.6]{Chen}
obtained for the case $x=0$:
\begin{equation}
H_{n}\big(B_{2k}\big(\tfrac{1}{2}\big)\big)=\prod_{j=1}^n \frac{\big((2j)!\big)^6}{(4j)!(4j+1)!}.\label{eq:B2k1Over2}
\end{equation}

\begin{table}[ht]
	\def\arraystretch{1.5}
	\centering
	\caption{$H_n\left(B_{2k}\left(\frac{1+x}{2}\right)\right)$ for $n=1,2,3$}\label{tab:Even}
	\scalebox{0.83}{%
	\begin{tabular}{c||c}
		\hline 
		$H_1\left(B_{2k}\left(\frac{1+x}{2}\right)\right)$ & $-\frac{1}{12} x^2+\frac{1}{45}$\\
		\hline 
		$H_2\left(B_{2k}\left(\frac{1+x}{2}\right)\right)$ & $-\frac{1}{540} x^6+\frac{97}{18900} x^4-\frac{11}{4725} x^2+\frac{16}{55125}$\\
  		\hline 
		$H_3\left(B_{2k}\left(\frac{1+x}{2}\right)\right)$ & $\frac{1}{42000}x^{12}-\frac{121}{441000}x^{10}+\frac{153}{154000}x^8-\frac{17441 }{12262250}x^6+\frac{8369 }{11036025}x^4-\frac{1632 }{9634625}x^2+\frac{256}{18883865}$\\
		\hline 
	\end{tabular}
	}
\end{table}


Although it seems that a closed expression for $H_{n}\big(B_{2k}(\frac{1+x}{2})\big)$ is out of reach, we shall show in this paper that the degree and leading coefficient of each Hankel determinant for even-indexed Bernoulli polynomials can be characterized, via a curious binomial transform \eqref{eq:B2kBn+k}.

\begin{theorem}
\label{prop:B2k}
For every $n\ge 0$, $H_{n}\big(B_{2k}(\frac{1+x}{2})\big)$ is a polynomial in $x$ of degree $n(n+1)$ with leading coefficient
\begin{equation}
\big[x^{n(n+1)}\big]H_{n}\left(B_{2k}\left(\frac{1+x}{2}\right)\right)=(-1)^{\binom{n+1}{2}}\prod_{j=1}^{n}\frac{(j!)^{6}}{(2j)!(2j+1)!}.\label{eq:B2k}
\end{equation}
\end{theorem}

\begin{remark}
	It is notable that the Hankel determinant $H_{n}\big(B_{2k}(\frac{1+x}{2})\big)$ is a linear combination of the terms
	\begin{align*}
		B_{2j_0}\left(\frac{1+x}{2}\right)B_{2(j_1+1)}\left(\frac{1+x}{2}\right)\cdots B_{2(j_n+n)}\left(\frac{1+x}{2}\right)
	\end{align*}
	with $j_0,j_1,\ldots,j_n$ a permutation of $0,1,\ldots,n$. Here, each term is of degree $\sum_{i=0}^n 2(j_i+i) = 2n(n+1)$. However, Theorem \ref{prop:B2k} states that the degree of the Hankel determinant $H_{n}\big(B_{2k}(\frac{1+x}{2})\big)$ is $n(n+1)$, which is only \textbf{half} of the above terms, thereby indicating abundant cancelations of higher powers of $x$ in this determinant expansion.
\end{remark}

As for generalizations, it is always of great interest and importance to consider the $q$-analogue of a number or polynomial sequence. For example, Carlitz \cite{Carlitz} introduced the $q$-Bernoulli numbers, denoted by $\beta_{k}$, as
\begin{align*}
	\beta_{k}:=\frac{1}{(1-q)^{k}}\sum_{j=0}^{k}(-1)^{j}\binom{k}{j}\frac{j+1}{[j+1]_{q}}.
\end{align*}
In 2017, Chapoton and Zeng \cite[p.~359, Eq.~(4.7)]{ChapotonZeng} showed that
\begin{align*}
	H_{n}(\beta_{k})=(-1)^{\binom{n+1}{2}}q^{\binom{n+1}{3}}\prod_{j=1}^{n}\frac{\left([j]_{q}!\right)^{6}}{[2j]_{q}![2j+1]_{q}!}.
\end{align*}
Here, the \emph{$q$-integers} for $m\in \mathbb{Z}$ and the \emph{$q$-factorials} for $M\in \mathbb{N}$ are respectively defined by 
\begin{align*}
	[m]_{q}:=\frac{1-q^{m}}{1-q}\qquad\text{and}\qquad [M]_{q}!:=\prod_{m=1}^{M}[m]_{q}.
\end{align*}
More recently, a similar study \cite{ChernJiu} was executed on the $q$-Euler numbers. 

Instead of seeking new $q$-analogues of known evaluations of Hankel determinants, in this work, we keep an eye on $q$-analogues of the binomial transform \eqref{eq:BinomialTransform}. Recall that the \emph{$q$-binomial coefficients} are defined by 
\begin{align*}
	\qbinom{n}{k}_{q}:=\begin{cases}
		\dfrac{(q;q)_n}{(q;q)_k(q;q)_{n-k}}, & \text{if $0\le k\le n$},\\[10pt]
		0, & \text{otherwise},
	\end{cases}
\end{align*}
where the \emph{$q$-Pochhammer symbols} for $N\in\mathbb{N}\cup\{\infty\}$ are defined by 
\begin{align*}
	(A;q)_{N}:=\prod_{k=0}^{N-1}(1-Aq^{k}),
\end{align*}
with the compact notation 
\begin{align*}
	(A_{1},A_{2},\ldots,A_{\ell};q)_{N}:=(A_{1};q)_{N}(A_{2};q)_{N}\cdots(A_{\ell};q)_{N}.
\end{align*}
It is notable that when the ordinary binomial coefficients $\binom{n}{k}$ in \eqref{eq:BinomialTransform} are replaced with the $q$-binomial coefficients $\qbinom{n}{k}_q$, the invariance \eqref{eq:inv} is no longer preserved. Now our objective is to characterize the leading coefficient for this $q$-binomial scenario.

Throughout, we assume that $(\alpha_{k})_{k\ge 0}$ is an arbitrary sequence. Define the polynomial sequence  
\begin{align}\label{eq:alpha-poly-1}
	\alpha_{k}(x):=\sum_{\ell=0}^{k}q^{\binom{\ell}{2}}\qbinom{k}{\ell}_{q}\alpha_{k-\ell}x^{\ell}.
\end{align}

\begin{theorem}\label{th:q-alpha-1}
For every $n\ge 0$, the Hankel determinant $H_{n}(\alpha_{k}(x))$ is a polynomial in $x$ of degree $n(n+1)$ with leading coefficient
\begin{equation}\label{eq:LCc1}
\big[x^{n(n+1)}\big]H_{n}(\alpha_{k}(x))=\alpha_{0}^{n+1}(-1)^{\binom{n+1}{2}}q^{3\binom{n+1}{3}}\prod_{j=1}^{n}(1-q^{j})^{n+1-j}.
\end{equation}
\end{theorem}

Clearly, the binomial transform \eqref{eq:BinomialTransform} is symmetric in the sense that
\begin{align*}
	a_{k}(x)=\sum_{\ell=0}^{k}\binom{k}{\ell}a_{k-\ell}x^{\ell} = \sum_{\ell=0}^{k}\binom{k}{\ell}a_{\ell}x^{k-\ell}.
\end{align*}
However, due to the prefactor $q^{\binom{\ell}{2}}$ (whose occurrence would be explained in Section \ref{sec:QBinomial}) in the summand of the $q$-binomial transform \eqref{eq:alpha-poly-1}, such a symmetry disappears. Hence, one may also want to consider
\begin{align}\label{eq:alpha-poly-2}
	\widetilde{\alpha}_{k}(x):=\sum_{\ell=0}^{k}q^{\binom{\ell}{2}}\qbinom{k}{\ell}_{q}\alpha_{\ell}x^{k-\ell}.
\end{align}

\begin{theorem}\label{th:LCc2}
	For every $n\ge 0$, the Hankel determinant $H_{n}(\widetilde{\alpha}_{k}(x))$ is a polynomial in $x$ of degree $\frac{n(n+1)}{2}$ with leading coefficient
	\begin{equation}\label{eq:LCc2}
		\big[x^{\frac{n(n+1)}{2}}\big]H_{n}(\widetilde{\alpha}_{k}(x))=\alpha_{0}\alpha_{1}\cdots \alpha_{n}(-1)^{\binom{n+1}{2}}q^{2\binom{n+1}{3}}\prod_{j=1}^{n}(1-q^{j})^{n+1-j}.
	\end{equation}
\end{theorem}

\begin{remark}
	The degree of $H_{n}(\alpha_{k}(x))$ is \emph{not} unexpected as in its determinant expansion, terms are of the form $\alpha_{j_0+0}(x)\alpha_{j_1+1}(x)\cdots \alpha_{j_n+n}(x)$ with $j_0,j_1,\ldots,j_n$ a permutation of $0,1,\ldots,n$, and each is of degree $n(n+1)$. Hence, $H_{n}(\alpha_{k}(x))$ has the \textbf{same} degree as these terms. However, the situation becomes different for $H_{n}(\widetilde{\alpha}_{k}(x))$ as its degree is only $\frac{n(n+1)}{2}$, being \textbf{half} of any term in its expansion. This fact places Theorem \ref{th:LCc2} in the same boat as Theorem \ref{prop:B2k}, and similar cancelations of higher powers of $x$ should happen.
\end{remark}

This paper is organized as follows.
In Section \ref{sec:Preliminaries}, we introduce some preliminary notation and lemmas for our subsequent use. In Section \ref{sec:EvenIndexed}, we establish a general transform for Bernoulli polynomials by means of the theory of umbral calculus, thereby offering a proof of Theorem \ref{prop:B2k}. In Section \ref{sec:QBinomial}, we characterize the leading coefficient in the Hankel determinants related to the $q$-binomial transforms \eqref{eq:alpha-poly-1} and \eqref{eq:alpha-poly-2} by proving Theorems \ref{th:q-alpha-1} and \ref{th:LCc2}. Meanwhile, we specialize our sequence $(\alpha_{k})_{k\ge 0}$ and provide an example with explicit Hankel determinant expressions. 

\section{Preliminaries}\label{sec:Preliminaries}

In this section, we shall collect some preliminaries for later use. 
\subsection{Determinants}

We need the following two basic relations on Hankel determinants.
\begin{lemma}
Let $(a_{k})_{k\ge 0}$ be a sequence and $x$ an indeterminate. Then for $n\geq0$, 
\begin{align}
	H_{n}(xa_{k}) & = x^{n+1}H_{n}(a_{k}),\label{eq:xDet}\\
	H_{n}(x^{k}a_{k}) & = x^{n(n+1)}H_{n}(a_{k}).\label{eq:xkDet}
\end{align}
\end{lemma}

These can be proved easily by extracting powers of $x$ from rows or columns. In particular, the relation \eqref{eq:xkDet} appeared as \cite[p.~4, Lemma 2.1]{DilcherJiu}.

\subsection{The Bernoulli umbra $\mathcal{B}$. }

The \emph{umbral calculus} (see, e.g., \cite{Roman}) associates each term of a sequence $(a_k)_{k\ge 0}$ with the corresponding power $\mathcal{A}^k$ of an \emph{umbra} $\mathcal{A}$. Along this line, we may define the \emph{Bernoulli umbra} $\mathcal{B}$ via a simple evaluation map:
\begin{align*}
	(\mathcal{B}+x)^{k}=\eval\left[(\mathcal{B}+x)^{k}\right]=B_{k}(x).
\end{align*}
For simplicity, the functional ``$\eval$'', short for evaluation, is usually omitted. Its probabilistic formalism \cite[Theorem 2.3]{Dixit} begins with a random variable $L_{B}$ subject to the density function $\frac{\pi}{2}\sech^{2}(\pi x)$. Then $\mathcal{B}=iL_{B}-\frac{1}{2}$, where $i=\sqrt{-1}$ is the imaginary unit. That is,
\begin{align*}
	B_k(x)=\mathbb{E}\left[(\mathcal{B}+x)^{k}\right]=\eval\left[(\mathcal{B}+x)^{k}\right]=\frac{\pi}{2}\int_{\mathbb{R}}\left(x+i\,t-\frac{1}{2}\right)^{k}\sech^{2}(\pi t)\,dt,
\end{align*}
which interprets the evaluation as the expectation operator $\mathbb{E}$, and $B_{k}(x)$ as the $k$-th moment of the random variable $iL_{B}-\frac{1}{2}+x$. For simplification, we shall just write 
\begin{equation}
\mathcal{B}^{k}=B_{k}\quad\text{and}\quad(\mathcal{B}+x)^{k}=B_{k}(x).\label{eq:BernoulliUmbral}
\end{equation}

\begin{remark}
The above random variable interpretation triggers the probabilistic proof in \cite{JiuShi} for the invariance \eqref{eq:inv} of Hankel determinants under the binomial transform. 
\end{remark}

%
%

%

\subsection{Orthogonal polynomials}

Orthogonal polynomials play a crucial role in the evaluation of Hankel determinants. See Ismail's monograph \cite[Chapter 2]{Ismail} or Krattenthaler's surveys \cite[Section 2.7]{Krattenthaler} and \cite[Section 5.4]{Kra2005} for a comprehensive account. Here we only collect some basics.

\begin{definition}
Given a sequence $(a_{n})_{n\ge 0}$, the family of (monic) polynomials $(p_{n}(z))_{n\ge 0}$ with $\deg p_{n}=n$ is called \emph{orthogonal with respect to $(a_{n})_{n\ge 0}$}, if there exists a linear functional $L$ on the space of polynomials in $z$ such that 
\begin{align*}
	L(z^{n})=a_{n}\quad \text{and} \quad L\big(p_{n}(z)p_{m}(z)\big)=\delta_{n,m}\sigma_{n},
\end{align*}
where $\delta_{n,m}$ is the \emph{Kronecker delta} and $(\sigma_{n})_{n\ge 0}$ is a sequence of nonzero constants. 
\end{definition}


In view of \emph{Favard's Theorem} \cite[p.~21, Theorem 12]{Krattenthaler}, there is a one-to-one correspondence between orthogonal polynomials and three-term recurrences.

\begin{lemma}[Favard's Theorem]
Let $(p_{n}(z))_{n\ge 0}$ be a family of monic polynomials with $p_{n}(z)$ of degree $n$. Then they are orthogonal if and only if there exist sequences $(u_{n})_{n\ge 0}$ and $(v_{n})_{n\ge 1}$ with $v_{n}$ nonzero such that $p_{0}(z)=1$, $p_{1}(z)=u_{0}+z$, and for $n\ge1$, 
\begin{equation}
p_{n+1}(z)=(u_{n}+z)p_{n}(z)-v_{n}p_{n-1}(z).\label{eq:3Term}
\end{equation}
\end{lemma}

The following relation \cite[p.~197, Theorem 51.1]{Wall}, which is attributed to Heilermann \cite{Hei1846}, further connects the Hankel determinants for a sequence with its associated orthogonal polynomials.

\begin{lemma}
Given any sequence $(a_{n})_{n\ge 0}$, let the monic polynomials $(p_{n}(z))_{n\ge 0}$ be orthogonal with respect to $(a_{n})_{n\ge 0}$. Assume further that $(p_{n}(z))_{n\ge 0}$ satisfies the recurrence in \eqref{eq:3Term}. Then
\begin{equation}
H_{n}(a_{k})=a_{0}^{n+1}v_{1}^{n}v_{2}^{n-1}\cdots v_{n}.\label{eq:HankelDet}
\end{equation}
\end{lemma}

Finally, we recall the \emph{big $q$-Jacobi polynomials} introduced by Andrews and Askey \cite{AndrewsAskey}, which form a family of basic hypergeometric orthogonal polynomials in the basic Askey scheme. To begin with, we need the \emph{$q$-hypergeometric series}:
\begin{align*}
	{}_{r+1}\phi_{r}\left(\begin{matrix}a_{1},a_{2},\ldots,a_{r+1}\\b_{1},b_{2},\ldots,b_{r}
	\end{matrix};q,t\right):=\sum_{n=0}^{\infty}\frac{(a_{1},a_{2},\ldots,a_{r+1};q)_{n}t^{n}}{(q,b_{1},b_{2},\ldots,b_{r};q)_{n}}.
\end{align*}

\begin{definition}
The \emph{big $q$-Jacobi polynomials} are defined
by 
\begin{equation}
P_{n}(z;a,b,c;q):={}_{3}\phi_{2}\left(\begin{matrix}q^{-n},abq^{n+1},z\\
aq,cq
\end{matrix};q,q\right).\label{eq:BigQJacobi}
\end{equation}
\end{definition}

The associated three-term recurrence for the big $q$-Jacobi polynomials is well-known; see, e.g., \cite[p.~438, Eq.~(14.5.3)]{KLS2010}.

\begin{lemma}\label{le:Jacobi-rec}
For $n\ge 1$,
\begin{align}
    (z-1)P_{n}(z;a,b,c;q)&=A_{n}P_{n+1}(z;a,b,c;q)-(A_{n}+C_{n})P_{n}(z;a,b,c;q)\notag\\
    &\quad+C_{n}P_{n-1}(z;a,b,c;q),
\end{align}
where
\begin{align*}
	A_{n}= & \frac{(1-aq^{n+1})(1-cq^{n+1})(1-abq^{n+1})}{(1-abq^{2n+1})(1-abq^{2n+2})},\\
	C_{n}= & -acq^{n+1}\frac{(1-q^{n})(1-bq^{n})(1-abc^{-1}q^{n})}{(1-abq^{2n})(1-abq^{2n+1})}.
\end{align*}
\end{lemma}
From \cite[p.~438, Eq.~(14.5.4)]{KLS2010}, we find the recurrence for the normalized sequence 
\begin{equation*}
    Q_{n}(z;a,b,c;q):=\frac{(aq,cq;q)_{n}}{(abq^{n+1};q)_{n}}P_{n}(z;a,b,c;q).
\end{equation*}
\begin{corollary}
For $n\ge 1$,
\begin{align}
    Q_{n+1}(z;a,b,c;q)&=(z+A_{n}+C_{n}-1)Q_{n}(z;a,b,c;q)\notag\\
    &\quad -A_{n-1}C_{n}Q_{n-1}(z;a,b,c;q).\label{eq:3TermQn}
\end{align}
\end{corollary}

In view of the orthogonality of the big $q$-Jacobi polynomials, we define the following sequence.

\begin{definition}
    Let $(\mu_{n}(a,b,c))_{n\ge 0}$ with $\mu_{0}(a,b,c)=1$ be a sequence with respect to which $P_{n}(z;a,b,c;q)$ is orthogonal.
\end{definition}

\begin{remark}\label{rmk:mu-Q}
By definition, the normalized polynomials $Q_{n}(z;a,b,c;q)$ are also orthogonal with respect to $(\mu_{n}(a,b,c))_{n\ge 0}$.
\end{remark}

To close this section, we note from \cite[p.~529, Theorem 3.1]{IsmailStanton} that $\mu_{n}(a,b,c)$ can be expressed in terms of the \emph{Al-Salam--Chihara
polynomials} \cite[Section 14.8]{KLS2010}.

\begin{definition}
The \emph{Al-Salam--Chihara polynomials} are defined by
\begin{equation}
p_{n}(\theta,t_{1},t_{2}):={}_{3}\phi_{2}\left(\begin{matrix}q^{-n},t_{1}e^{i\theta},t_{2}e^{-i\theta}\\
t_{1}t_{2},0
\end{matrix};q,q\right).\label{eq:Al-Salam-Chihara-Polynomial}
\end{equation}
\end{definition}

\begin{lemma}
\label{thm:MomentsBigQISAl-Salam-Chihara}
Let $a=t_{1}e^{i\theta}/q$,
$b=t_{2}e^{-i\theta}/q$, and $c=t_{1}e^{-i\theta}/q$. Then
\begin{align}
    \mu_{n}(a,b,c)=p_{n}(\theta,t_{1},t_{2}).
\end{align}
\end{lemma}

\section{Even-indexed Bernoulli polynomials}\label{sec:EvenIndexed}

Our starting point is the following identity for Bernoulli polynomials.

\begin{lemma}
	For any indeterminates $a$ and $b$,
	\begin{equation}\label{eq:B2kBn+k}
		\sum_{k=0}^{n}\binom{n}{k}B_{2k}(ax+b)(-b^{2})^{n-k}=\sum_{k=0}^{n}\binom{n}{k}B_{n+k}(ax)(2b)^{n-k}.
	\end{equation}
\end{lemma}

\begin{proof}
	We make use of the Bernoulli umbra $\mathcal{B}$ and derive that
	\begin{align*}
		\sum_{k=0}^{n}\binom{n}{k}B_{2k}(ax+b)(-b^{2})^{n-k}  &=\sum_{k=0}^{n}\binom{n}{k}(\mathcal{B}+ax+b)^{2k}(-b^{2})^{n-k}\\
		 &=\left((\mathcal{B}+ax+b)^{2}-b^{2}\right)^{n} \\
		 &=(\mathcal{B}+ax)^{n}(\mathcal{B}+ax+2b)^{n} \\
   &=\sum_{k=0}^{n}\binom{n}{k}(\mathcal{B}+ax)^{n+k}(2b)^{n-k} \\
   &=\sum_{k=0}^{n}\binom{n}{k}B_{n+k}(ax)(2b)^{n-k}.
	\end{align*}
	In particular, for the first equality, $B_{2k}(ax+b)$ is replaced with $(\mathcal{B}+ax+b)^{2k}$, and for the last equality, $(\mathcal{B}+ax)^{n+k}$ is replaced with $B_{n+k}(ax)$, both by means of \eqref{eq:BernoulliUmbral}.
\end{proof}

In \cite[p.~390, Eq.~(42)]{Chen}, Chen established the following relation on the Hankel determinants for the \emph{median Bernoulli numbers} $K_k$,
\begin{equation}\label{eq:KB}
	H_{n}(K_{k})=\big({-\tfrac{1}{2}}\big)^{n+1}H_{n}\big(B_{2k}\big(\tfrac{1}{2}\big)\big),
\end{equation}
wherein
\begin{align*}
	K_{k}:=-\frac{1}{2}\sum_{\ell=0}^{k}\binom{k}{\ell}B_{k+\ell}.
\end{align*}
Now, \eqref{eq:B2kBn+k} gives a generalization of Chen's relation \eqref{eq:KB} as follows.

\begin{corollary}
	For any indeterminates $a$ and $b$, define 
	\begin{align*}
		K_{k}^{a,b}(x):=\sum_{\ell=0}^{k}\binom{k}{\ell}B_{k+\ell}(ax)(2b)^{k-\ell}.
	\end{align*}
	Then
	\begin{equation}
		H_{n}(K_{k}^{a,b}(x))=H_{n}(B_{2k}(ax+b)).\label{eq:KabxB}
	\end{equation}
\end{corollary}

\begin{proof}
    Note that \eqref{eq:B2kBn+k} allows us to rewrite $K_{k}^{a,b}(x)$ as
    \begin{align*}
        K_{k}^{a,b}(x) = \sum_{\ell=0}^{k}\binom{k}{\ell}B_{2\ell}(ax+b)(-b^2)^{k-\ell}.
    \end{align*}
    Applying Theorem \ref{prop:BinomialTransform} with the specialization $a_k\mapsto B_{2k}(ax+b)$ and $x\mapsto -b^2$ yields the desired evaluation.
\end{proof}

It is an easy observation that to show Theorem \ref{prop:B2k}, especially \eqref{eq:B2k}, it suffices to confirm that the leading coefficient of $H_{n}\big(B_{2k}(\frac{1+x}{2})\big)$ equals $H_{n}(B_{k})$ by recalling \eqref{eq:Bk}. Now our objective is to prove this claim, thereby offering an assertion of Theorem \ref{prop:B2k}.

\begin{proof}[Proof of Theorem \ref{prop:B2k}]
	Letting $a\mapsto \frac{1}{2x}$ and $b\mapsto \frac{1}{2x}$ in \eqref{eq:KabxB}, we have 
	\begin{align*}
		H_{n}\left(B_{2k}\left(\frac{1+x^{-1}}{2}\right)\right)=H_{n}\left(\sum_{\ell=0}^{k}\binom{k}{\ell}B_{k+\ell}\big(\tfrac{1}{2}\big)\left(\frac{1}{x}\right)^{k-\ell}\right).
	\end{align*}
	Note that for the two sequences
	\begin{align*}
		\left(B_{2k}\left(\frac{1+x^{-1}}{2}\right)\right)_{k\ge 0} \qquad\text{and}\qquad \left(\sum_{\ell=0}^{k}\binom{k}{\ell}B_{k+\ell}\big(\tfrac{1}{2}\big)\left(\frac{1}{x}\right)^{k-\ell}\right)_{k\ge 0},
	\end{align*}
	which share the same Hankel determinants, if we simultaneously multiply the $k$-th term of them by $x^{k}$, then \eqref{eq:xkDet} asserts that the new Hankel determinants are still identical. Thus,
	\begin{align*}
		H_{n}\left(\sum_{\ell=0}^{k}\binom{k}{\ell}B_{k+\ell}\left(\frac{1}{2}\right)\,x^{\ell}\right)&=H_{n}\left(x^{k}B_{2k}\left(\frac{1+x^{-1}}{2}\right)\right)\\
        &=x^{n(n+1)}H_{n}\left(B_{2k}\left(\frac{1+x^{-1}}{2}\right)\right).
	\end{align*}
	We then observe that $H_{n}\big(B_{2k}\big(\frac{1+x^{-1}}{2}\big)\big)$ are \emph{polynomials} in $x^{-1}$. Assuming that the degree of $H_{n}\big(B_{2k}\big(\frac{1+x}{2}\big)\big)$ is $d>n(n+1)$, 
	then in $x^{n(n+1)}H_{n}\big(B_{2k}\big(\frac{1+x^{-1}}{2}\big)\big)$, there exists a non-vanishing \emph{negative} power of $x$, that is, $x^{n(n+1)-d}$. However, the Hankel determinants $H_{n}\big(\sum_{j=0}^{k}\binom{k}{j}B_{k+j}\big(\tfrac{1}{2}\big)x^{j}\big)$ are \emph{polynomials} in $x$, and so are $x^{n(n+1)}H_{n}\big(B_{2k}\big(\frac{1+x^{-1}}{2}\big)\big)$, which should \emph{not} contain any negative power of $x$. We are led to a contradiction, and hence, the degree of $H_{n}\big(B_{2k}\big(\frac{1+x}{2}\big)\big)$ is \emph{at most} $n(n+1)$. Finally, we compute the coefficient at $x^{n(n+1)}$ as follows:
	\begin{align*}
		\big[x^{n(n+1)}\big]H_{n}\left(B_{2k}\left(\frac{1+x}{2}\right)\right) & =  \big[x^0\big]x^{n(n+1)}H_{n}\left(B_{2k}\left(\frac{1+x^{-1}}{2}\right)\right)\\
		& = \big[x^0\big] H_{n}\left(\sum_{\ell=0}^{k}\binom{k}{\ell}B_{k+\ell}\big(\tfrac{1}{2}\big)\,x^{\ell}\right)\\
		& = H_{n}\left(\sum_{\ell=0}^{k}\binom{k}{\ell}B_{k+\ell}\big(\tfrac{1}{2}\big)\,x^{\ell}\bigg|_{x=0}\right).
	\end{align*}
	Thus,
	\begin{align*}
		\big[x^{n(n+1)}\big]H_{n}\left(B_{2k}\left(\frac{1+x}{2}\right)\right) = H_{n}\big(B_{k}\big(\tfrac{1}{2}\big)\big) \overset{\text{\eqref{eq:Hankel-B-poly}}}{=} H_{n}(B_{k}).
	\end{align*}
	In particular, this value is non-vanishing in view of \eqref{eq:Bk}, thereby implying that the degree of $H_{n}\big(B_{2k}\big(\frac{1+x}{2}\big)\big)$ is \emph{exactly} $n(n+1)$.
\end{proof}

\section{$q$-Binomial transforms}\label{sec:QBinomial}


To understand the underlying motivation of our $q$-binomial transform \eqref{eq:alpha-poly-1}, we recall the theory of umbral calculus and define an umbra $\mathcal{A}$ with evaluation $\mathcal{A}^{k}=a_{k}$ in order to rewrite the binomial transform \eqref{eq:BinomialTransform} as
\begin{align*}
	a_k(x) = \sum_{\ell=0}^{k}\binom{k}{\ell}a_{k-\ell}x^{\ell} = \sum_{\ell=0}^{k}\binom{k}{\ell}\mathcal{A}^{k-\ell}x^{\ell} = \prod_{j=0}^{k-1} (\mathcal{A}+x),
\end{align*}
wherein the last equality follows from the \emph{binomial theorem}, which can be further generalized as the \emph{$q$-binomial theorem} \cite[p.~421, Eq.~(17.2.35)]{And2010}:
\begin{equation*}
	\prod_{k=0}^{n-1}(q^{k}\alpha+x)=\sum_{k=0}^{n}q^{\binom{k}{2}}\qbinom{n}{k}_{q}\alpha^{n-k}x^{k}.
\end{equation*}
Thus, we similarly define an umbra $\boldsymbol{\alpha}$ with evaluation $\boldsymbol{\alpha}^{k}=\alpha_{k}$ for $(\alpha_{k})_{k\ge 0}$. Then the $q$-binomial transform \eqref{eq:alpha-poly-1} is reformulated as
\begin{align*}
	\alpha_{k}(x)=\sum_{\ell=0}^{k}q^{\binom{\ell}{2}}\qbinom{k}{\ell}_{q}\alpha_{k-\ell}x^{\ell} = \sum_{\ell=0}^{k}q^{\binom{\ell}{2}}\qbinom{k}{\ell}_{q}\boldsymbol{\alpha}^{k-\ell}x^{\ell} = \prod_{j=0}^{k-1} (q^j \boldsymbol{\alpha} + x).
\end{align*}
The above discussion explains the occurrence of the prefactor $q^{\binom{\ell}{2}}$ in the summand.

\subsection{Leading coefficient}

In this section, our main objective is to analyze the Hankel determinants $H_n(\alpha_k(x))$ and $H_{n}(\widetilde{\alpha}_{k}(x))$, especially the leading coefficient in them as stated in Theorems \ref{th:q-alpha-1} and \ref{th:LCc2}.

\subsubsection{The Hankel determinants $H_n(\alpha_k(x))$}

It is clear that our $q$-binomial transform \eqref{eq:alpha-poly-1} reduces to the binomial transform \eqref{eq:BinomialTransform} by taking the $q\to 1$ limit. However, unlike the invariance relation \eqref{eq:inv} for the binomial case, there is no obvious pattern connecting $H_n(\alpha_k(x))$ and $H_n(\alpha_k)$. In Table~\ref{tab:CoefficientsH2}, we have computed the polynomial expansion of $H_2(\alpha_k(x))$, which is of degree $6$, and one may see the chaos among these coefficients.
\begin{table}[ht]
	\def\arraystretch{1.5}
	\centering
	\caption{Coefficients in $H_2(\alpha_k(x))$}\label{tab:CoefficientsH2}
	\scalebox{0.83}{%
	\begin{tabular}{c||c}
		\hline 
		$x^{6}$ & $-q^{3}(1-q)^{3}(1+q)\alpha_{0}^{3}$\\
		\hline 
		$x^{5}$ & $-q^{2}(1-q)^{3}(1+q)(1+2q)\alpha_{0}^{2}\alpha_{1}$\\
		\hline 
		$x^{4}$ & $-q(1-q)^{2}(1+q)\big[(1+2q+q^2-q^3)\alpha_{0}\alpha_{1}^{2}-(1+2q^2)\alpha_{0}^2\alpha_{2}\big]$\\
		\hline 
		$x^{3}$ & { $(1-q)^{2}(1+q) \big[(2+q+q^{2}+2q^{3})\alpha_{0}\alpha_{1}\alpha_{2}-(1-q)\alpha_{0}^{2}\alpha_{3}-(1+2q+2q^2+q^3)\alpha_{1}^{3}\big]$}\\
		\hline 
		$x^{2}$ & {\small $(1-q) \big[(1+3q+q^2-q^3)\alpha_0\alpha_1\alpha_3+(1-q-q^2-q^3-q^4)\alpha_0\alpha_2^2-\alpha_0^2\alpha_4-(1+2q-2q^3-q^4)\alpha_1^2\alpha_2\big]$}\\
		\hline 
		$x^{1}$ & $-(1-q)\big[\alpha_0\alpha_1\alpha_4-(1-q^2)\alpha_0\alpha_2\alpha_3+(1+2q)\alpha_1\alpha_2^2-(1+2q+q^2)\alpha_1^2\alpha_3\big]$\\
		\hline 
		$x^{0}$ & $\alpha_{0}\alpha_{2}\alpha_{4}-\alpha_{0}\alpha_{3}^{2}+2\alpha_{1}\alpha_{2}\alpha_{3}-\alpha_{1}^{2}\alpha_{4}-\alpha_{2}^{3}$\\
		\hline 
	\end{tabular}
	}
\end{table}

 In this subsection, we explore the behavior of the leading coefficient in $H_n(\alpha_k(x))$, and prove Theorem \ref{th:q-alpha-1}. To begin with, we need the Hankel determinant evaluation for the sequence $\big(q^{\binom{k}{2}}\big)_{k\ge 0}$.
\begin{lemma}\label{le:HDetQBinomialk2}
	For $n\ge 0$,
	\begin{align}
		H_{n}\big(q^{\binom{k}{2}}\big) =(-1)^{\binom{n+1}{2}}q^{3\binom{n+1}{3}}\prod_{j=1}^{n}(1-q^{j})^{n+1-j}.\label{eq:HDetQBinomialk2}
	\end{align}
\end{lemma}

\begin{proof}
	Note that $H_{n}\big(q^{\binom{k}{2}}\big)$ is the determinant of the Hankel matrix $\big(q^{\binom{i+j}{2}}\big)_{i,j=0}^n$. For its $(i,j)$-th entry, we may rewrite it as
	\begin{align*}
		q^{\binom{i+j}{2}} = q^{\binom{i}{2}}\cdot q^{\binom{j}{2}}\cdot q^{ij}.
	\end{align*}
	Hence, we first factor out $q^{\binom{i}{2}}$ from the $i$-th row for each $i$, and then $q^{\binom{j}{2}}$ from the $j$-th column for each $j$. Then
	\begin{align*}
		H_{n}\big(q^{\binom{k}{2}}\big)  = q^{\sum_{i=0}^n \binom{i}{2}}\cdot q^{\sum_{j=0}^n \binom{j}{2}}\cdot \underset{{0\le i,j\le n}}{\det}(q^{ij})= q^{2\binom{n+1}{3}}\cdot \prod_{0\le i< j\le n} (q^j-q^i),
	\end{align*}
	where we have applied the evaluation for the Vandermonde determinant \cite[p.~366, Theorem 9.67]{Axl2023} in the last equality. Finally, we note that
	\begin{align*}
		\prod_{0\le i< j\le n} (q^j-q^i) &= (-1)^{\binom{n+1}{2}} \prod_{0\le i< j\le n} (q^i-q^j)\\
		& = (-1)^{\binom{n+1}{2}} \left(\prod_{i=0}^n (q^i)^{n-i}\right) \left(\prod_{0\le i<j\le n} (1-q^{j-i})\right)\\
		(\text{with $\ell := j-i$}) & =  (-1)^{\binom{n+1}{2}} q^{\binom{n+1}{3}} \prod_{\ell=1}^n (1-q^\ell)^{n+1-\ell},
	\end{align*}
	thereby confirming the required identity.
\end{proof}

We are in a position to prove Theorem \ref{th:q-alpha-1}.

\begin{proof}[Proof of Theorem \ref{th:q-alpha-1}]
	We reflect the polynomials $\alpha_k(x)$ by defining for each $k\ge 0$,
	\begin{align}\label{eq:alpha-reflection}
		\alpha'_k(x) := x^k\alpha_k(x^{-1}) = \sum_{\ell=0}^{k}q^{\binom{\ell}{2}}\qbinom{k}{\ell}_{q}\alpha_{k-\ell}x^{k-\ell}.
	\end{align}
	In particular, by \eqref{eq:xkDet},
	\begin{align}\label{eq:H-alpha-beta}
		H_n(\alpha'_k(x)) = x^{n(n+1)} H_n(\alpha_k(x^{-1})).
	\end{align}
	Note that $\alpha_k(x)$ and $\alpha'_k(x)$ are \emph{polynomials} in $x$, and so are $H_n(\alpha_k(x))$ and $H_n(\alpha'_k(x))$. Hence, the degree of $H_n(\alpha_k(x))$ is \emph{at most} $n(n+1)$; otherwise, there exists a \emph{negative} power of $x$ in $x^{n(n+1)} H_n(\alpha_k(x^{-1})) = H_n(\alpha'_k(x))$, which is invalid. Finally, \eqref{eq:H-alpha-beta} implies that $\big[x^{n(n+1)}\big]H_{n}(\alpha_{k}(x))$ equals
	\begin{align*}
		\big[x^{0}\big]H_{n}(\alpha'_{k}(x)) = H_{n}(\alpha'_{k}(0)) = H_{n}\big(\alpha_0q^{\binom{k}{2}}\big) \overset{\text{\eqref{eq:xDet}}}{=} \alpha_0^{n+1} H_{n}\big(q^{\binom{k}{2}}\big).
	\end{align*}
	Invoking \eqref{eq:HDetQBinomialk2} gives \eqref{eq:LCc1}. Also, this value is non-vanishing, and hence, the degree of $H_n(\alpha_k(x))$ is \emph{exactly} $n(n+1)$.
\end{proof}

\subsubsection{The Hankel determinants $H_{n}(\widetilde{\alpha}_{k}(x))$}

Now, we move on to the second $q$-binomial transform \eqref{eq:alpha-poly-2}. It can be similarly seen from Table~\ref{tab:Coefficient-alpha-2} that the polynomial expansions of the Hankel determinants $H_{n}(\widetilde{\alpha}_{k}(x))$ behave chaotically.

\begin{table}[ht]
	\def\arraystretch{1.5}
	\centering
	\caption{Coefficients in $H_2(\widetilde{\alpha}_k(x))$}\label{tab:Coefficient-alpha-2}
	\scalebox{0.9}{%
		\begin{tabular}{c||c}
			\hline 
			$x^{3}$ & $-q^{2}(1-q)^{3}(1+q)\alpha_{0}\alpha_{1}\alpha_{2}$\\
			\hline 
			$x^{2}$ & $q^{2}(1-q)^{2}(1+q)\big[(q+q^2)\alpha_{0}\alpha_{1}\alpha_{3}-\alpha_{0}\alpha_{2}^2-\alpha_{1}^2\alpha_{2}\big]$\\
			\hline 
			$x^{1}$ & $-q^2(1-q)\big[q^4\alpha_{0}\alpha_{1}\alpha_{4}-(q^2-q^4)\alpha_{0}\alpha_{2}\alpha_{3}+(1+2q)\alpha_{1}\alpha_{2}^2-(q+2q^2+q^3)\alpha_{1}^2\alpha_{3}\big]$\\
			\hline 
			$x^{0}$ & $q^3 \big[q^4\alpha_{0}\alpha_{2}\alpha_{4}-q^3\alpha_{0}\alpha_{3}^2+2q\alpha_{1}\alpha_{2}\alpha_{3}-q^3\alpha_{1}^2\alpha_{4}-\alpha_{2}^3\big]$\\
			\hline 
		\end{tabular}
	}
\end{table}

As in \eqref{eq:alpha-reflection}, to prove Theorem \ref{th:LCc2}, we also need to reflect the polynomials $\widetilde{\alpha}_{k}(x)$ for each $k\ge 0$ as
\begin{align}\label{eq:alpha-2-reflection}
	\widetilde{\alpha}'_k(x) := x^k\widetilde{\alpha}_k(x^{-1}) = \sum_{\ell=0}^{k}q^{\binom{\ell}{2}}\qbinom{k}{\ell}_{q}\alpha_{\ell}x^{\ell}.
\end{align}
In light of \eqref{eq:xkDet},
\begin{align}\label{eq:H-alpha-2-beta}
	H_n(\widetilde{\alpha}'_k(x)) = x^{n(n+1)} H_n(\widetilde{\alpha}_k(x^{-1})).
\end{align}
Now we require some extra effort to proceed with our proof. To begin with, we define a family of modified difference operators $\Delta_i$ for arbitrary sequences $(s_k)_{k\ge 0}$ as follows:
\begin{align*}
	\Delta_i(s_k):=\begin{cases}
		s_k, & \text{if $i=0$},\\[6pt]
		s_k-q^{i-1}s_{k-1}, & \text{if $i\ge 1$}.
	\end{cases}
\end{align*}

\begin{lemma}
	For every $m\ge 0$,
	\begin{align}\label{eq-Claim}
		\Delta_m\Delta_{m-1}\cdots\Delta_0(\widetilde{\alpha}'_k(x))=\sum_{\ell=m}^k q^{(k-\ell)m+\binom{\ell}{2}}\qbinom{k-m}{\ell-m}_q \alpha_\ell x^\ell.
	\end{align}
\end{lemma}

\begin{proof}
	We prove this lemma by induction on $m$. For $m=0$, the relation \eqref{eq-Claim} simply reduces to \eqref{eq:alpha-2-reflection}. Suppose that \eqref{eq-Claim} holds for a certain integer $m\ge 0$. Then
	\begin{align*}
		&\Delta_{m+1}\big(\Delta_m\cdots \Delta_1(\widetilde{\alpha}'_k(x))  \big)\\
		&\qquad=q^{\binom{k}{2}}\alpha_k x^k+\sum_{\ell=m}^{k-1} q^{(k-\ell)m+\binom{\ell}{2}}\qbinom{k-m}{\ell-m}_q \alpha_\ell x^\ell \\
		&\qquad\quad -q^m\sum_{\ell=m}^{k-1} q^{(k-1-\ell)m+\binom{\ell}{2}}\qbinom{k-1-m}{\ell-m}_q \alpha_\ell x^\ell \\
		&\qquad=q^{\binom{k}{2}}\alpha_k x^k+\sum_{\ell=m}^{k-1} q^{(k-\ell)m+\binom{\ell}{2}}\left(\qbinom{k-m}{\ell-m}_q-\qbinom{k-1-m}{\ell-m}_q\right) \alpha_\ell x^\ell \\
		&\qquad=q^{\binom{k}{2}}\alpha_k x^k+\sum_{\ell=m+1}^{k-1} q^{(k-\ell)(m+1)+\binom{\ell}{2}}\qbinom{k-1-m}{\ell-1-m}_q \alpha_\ell x^\ell,
	\end{align*}
	where we have applied the following Pascal-type relation \cite[p.~35, Eq.~(3.3.3)]{And1998} for the last equality:
	\begin{align*}
		\qbinom{N}{M}_q-\qbinom{N-1}{M}_q = q^{N-M}\qbinom{N-1}{M-1}_q.
	\end{align*}
	Therefore, we have shown the $m+1$ case, and hence the relation \eqref{eq-Claim} in general.
\end{proof}

We are ready to prove Theorem \ref{th:LCc2}.

\begin{proof}[Proof of Theorem \ref{th:LCc2}]
	We first show that the degree of $H_{n}(\widetilde{\alpha}_{k}(x))$ is \emph{at most} $\frac{n(n+1)}{2}$. Note that in view of \eqref{eq:H-alpha-2-beta}, it is sufficient to prove that the degree of $H_{n}(\widetilde{\alpha}'_{k}(x))$ is \emph{at least} $\frac{n(n+1)}{2}$. To see this, we apply $n$ steps of row transformations to the determinant. Specifically speaking, for each $k=1,\ldots,n$, in the $k$-th step we apply the row transformation $r_i-q^{k-1}r_{i-1}$ to the $i$-th row for each $i=n,n-1,\dots,k$. Here, we should recall that our Hankel matrix is indexed from the \emph{zeroth} row. Now, the determinant becomes
	\begin{align*}
		H_{n}(\widetilde{\alpha}'_{k}(x))=\underset{{0\le i,j\le n}}{\det}\big(\Delta_i\Delta_{i-1}\cdots \Delta_{0}(\widetilde{\alpha}'_{i+j}(x))\big),
	\end{align*}
	whose $(i,j)$-entry can be further evaluated by \eqref{eq-Claim},
	\begin{align*}
		\Delta_i\Delta_{i-1}\cdots \Delta_{0}(\widetilde{\alpha}'_{i+j}(x)) & =\sum_{\ell=i}^{i+j} q^{(i+j-\ell)i+\binom{\ell}{2}}\qbinom{j}{\ell-i}_q \alpha_\ell x^\ell \\
		&=q^{ij+\binom{i}{2}}\alpha_i x^i+x^{i+1}R_{i,j}(x),
	\end{align*}
	where each $R_{i,j}(x)$ is a \emph{polynomial} in $x$ whose explicit form is not needed in our arguments. It turns out that entries in the $i$-th row have a common factor $x^{i}$. Extracting the factor $x^{i}$ from the $i$-th row for each $i=0,1,\ldots,n$, we deduce that
	\begin{align}\label{eq:alpha-2-det-new}
		H_{n}(\widetilde{\alpha}'_{k}(x))=x^{\frac{n(n+1)}{2}}\underset{{0\le i,j\le n}}{\det}\big(q^{ij+\binom{i}{2}}\alpha_i +xR_{i,j}(x)\big),
	\end{align}
	which implies that $\deg H_{n}(\widetilde{\alpha}'_{k}(x))\ge \frac{n(n+1)}{2}$, as requested.
 
 Finally, to show that the degree of $H_{n}(\widetilde{\alpha}_{k}(x))$ is \emph{exactly} $\frac{n(n+1)}{2}$, we only need to confirm that the corresponding coefficient is non-vanishing, as given in \eqref{eq:LCc2}:
	\begin{align*}
		\big[x^{\frac{n(n+1)}{2}}\big]H_{n}(\widetilde{\alpha}_{k}(x)) \overset{\text{\eqref{eq:H-alpha-2-beta}}}{=} \big[x^{\frac{n(n+1)}{2}}\big]H_{n}(\widetilde{\alpha}'_{k}(x)) \overset{\text{\eqref{eq:alpha-2-det-new}}}{=} \underset{{0\le i,j\le n}}{\det}\big(q^{ij+\binom{i}{2}}\alpha_i\big),
	\end{align*}
	which equals
	\begin{align*}
		\alpha_0\alpha_1\cdots \alpha_n q^{\binom{n+1}{3}} \underset{{0\le i,j\le n}}{\det}\big(q^{ij}\big) = \alpha_0\alpha_1\cdots \alpha_n (-1)^{\binom{n+1}{2}}q^{2\binom{n+1}{3}}\prod_{j=1}^{n}(1-q^{j})^{n+1-j},
	\end{align*}
	wherein the Vandermonde determinant has been evaluated in the proof of Lemma \ref{le:HDetQBinomialk2}.
\end{proof}

\subsection{An example admitting explicit Hankel determinant expressions}

For $q$-binomial transforms \eqref{eq:alpha-poly-1} and \eqref{eq:alpha-poly-2}, it is in general out of reach to establish closed expressions for the related Hankel determinants. However, some specific sequences $(\alpha_{k})_{k\ge 0}$ still admit a neat Hankel determinant evaluation. In this section, we provide such an instance.

\begin{theorem}
\label{prop:qBinomial1} 
Choose the sequence $\alpha_{k}^{u,v}:=q^{-\binom{k}{2}}(u;q)_{k}\,v^{k}$ with $u$ and $v$ indeterminates in \eqref{eq:alpha-poly-2}, and let
\begin{align*}
	\widetilde{\alpha}_{k}^{u,v}(x):=\sum_{\ell=0}^{k}\qbinom{k}{\ell}_{q}(u;q)_{\ell}\,v^{\ell}x^{k-\ell}=\sum_{\ell=0}^{k}\qbinom{k}{\ell}_{q}(u;q)_{k-\ell}\,v^{k-\ell}x^{\ell}.
\end{align*}
Then
\begin{equation}
	H_{n}(\widetilde{\alpha}_{k}^{u,v}(x))=v^{\binom{n+1}{2}}q^{\binom{n+1}{3}}\prod_{j=1}^{n}(uv q^{j-1}-x)^{n+1-j}(u,q;q)_{n+1-j}.\label{eq:FinalExample}
\end{equation}
\end{theorem}

\begin{proof}
	Recall that the \emph{second family of Al-Salam--Carlitz polynomials} \cite[p.~53, Eq.~(4.8)]{AlSalamCarlitz2}:
	\begin{align*}
		V_{k}^{(a)}(x;q):=(-1)^{k}q^{-\binom{k}{2}}\sum_{\ell=0}^{k}\qbinom{k}{\ell}_{q}a^{\ell}(x;q)_{k-\ell}
	\end{align*}
	admits the following generating function identity \cite[p.~531, Eq.~(3.22)]{IsmailStanton}:
	\begin{align*}
		\sum_{k=0}^{\infty}V_{k}^{(a)}(x;q)\frac{q^{\binom{k}{2}}}{(q;q)_{k}}(-t)^{k}=\frac{(xt;q)_{\infty}}{(t,at;q)_{\infty}}.
	\end{align*}
	Now letting $t\mapsto vw$, $a\mapsto x/v$, and $x\mapsto u$, we have 
	\begin{align*}
		\frac{(uvw;q)_{\infty}}{(vw,xw;q)_{\infty}}=\sum_{k=0}^{\infty}V_{k}^{(x/v)}(u;q)\frac{q^{\binom{k}{2}}}{(q;q)_{k}}(-vw)^{k},
	\end{align*}
	wherein it is further notable that
	\begin{align*}
		V_{k}^{(x/v)}(u;q)q^{\binom{k}{2}}(-v)^{k}=\sum_{\ell=0}^{k}\qbinom{k}{\ell}_{q}\frac{x^{\ell}}{v^{\ell}}(u;q)_{k-\ell}v^{k}=\widetilde{\alpha}_{k}^{u,v}(x).
	\end{align*}
	Thus, we have the generating function identity
	\begin{align}\label{ex-alpha-gf}
		\sum_{k=0}^{\infty}\frac{\widetilde{\alpha}_{k}^{u,v}(x)}{(q;q)_{k}}w^{k}=\frac{(uvw;q)_{\infty}}{(vw,xw;q)_{\infty}}.
	\end{align}
	
	In the meantime, note that the generating function for the Al-Salam--Chiahara polynomials \eqref{eq:Al-Salam-Chihara-Polynomial} can be found in \cite[p.~531, Eq.~(3.18)]{IsmailStanton}:
	\begin{align*}
		\sum_{k=0}^{\infty}\frac{(t_{1}t_{2};q)_{k}\,p_{k}(\theta;t_{1},t_{2})}{(q;q)_{k}t_{1}^{k}}t^{k}=\frac{(t_{1}t,t_{2}t;q)_{\infty}}{(te^{i\theta},te^{-i\theta};q)_{\infty}}.
	\end{align*}
	In light of Lemma \ref{thm:MomentsBigQISAl-Salam-Chihara}, we know that $p_{k}(\theta,t_{1},t_{2})=\mu_{k}(a,b,c)$ by taking $a= t_{1}e^{i\theta}/q$, $b= t_{2}e^{-i\theta}/q$, and $c= t_{1}e^{-i\theta}/q$. We further let $e^{i\theta}\mapsto \sqrt{v/x}$, $t_{1}\mapsto u\sqrt{v/x}$, and $t_{2}\mapsto 0$, which leads to 
	\begin{align}\label{eq:abc-value}
		(a,b,c)=\left(\frac{uv}{xq},0,\frac{u}{q}\right).
	\end{align}
	For this choice of $a$, $b$, and $c$,
	\begin{align*}
		\sum_{k=0}^{\infty}\frac{\mu_{k}(a,b,c)}{(q;q)_{k}}\left(\frac{t}{u\sqrt{v/x}}\right)^{k}=\frac{(ut\sqrt{v/x};q)_{\infty}}{(t\sqrt{v/x},t\sqrt{x/v};q)_{\infty}}.
	\end{align*}
	Now, making the substitution $w=t/\sqrt{vx}$ in the above yields
	\begin{align*}
		\sum_{k=0}^{\infty}\frac{\mu_{k}(a,b,c)}{(q;q)_{k}}\left(\frac{wx}{u}\right)^{k}=\frac{(uvw;q)_{\infty}}{(vw,xw;q)_{\infty}}\overset{\text{\eqref{ex-alpha-gf}}}{=} \sum_{k=0}^{\infty}\frac{\widetilde{\alpha}_{k}^{u,v}(x)}{(q;q)_{k}}w^{k}.
	\end{align*}
	This shows that for every $k\ge 0$,
	\begin{align}\label{eq:alpha-mu}
		\widetilde{\alpha}_{k}^{u,v}(x)=\left(\frac{x}{u}\right)^k \mu_{k}(a,b,c).
	\end{align}
	
	Now it remains to evaluate the Hankel determinants $H_{n}(\mu_{k}(a,b,c))$. In view of Remark \ref{rmk:mu-Q}, we combine \eqref{eq:3Term}, \eqref{eq:HankelDet}, and \eqref{eq:3TermQn} to get
	\begin{align*}
		H_{n}(\mu_{k}(a,b,c))=\prod_{j=1}^{n}(A_{j-1}C_{j})^{n+1-j},
	\end{align*}
	wherein the substitution of variables in \eqref{eq:abc-value} should be executed on the sequences $A_j$ and $C_j$ in Lemma \ref{le:Jacobi-rec}:
	\begin{align*}
		A_{j}= & \frac{(1-aq^{j+1})(1-cq^{j+1})(1-abq^{j+1})}{(1-abq^{2j+1})(1-abq^{2j+2})}=\left(1-\frac{uv q^{j}}{x}\right)(1-u q^{j}),\\
		C_{j}= & -\frac{(1-q^{j})(1-bq^{j})(1-abc^{-1}q^{j})}{(1-abq^{2j})(1-abq^{2j+1})}acq^{j+1}=-\frac{u^{2}vq^{j-1}(1-q^{j})}{x}.
	\end{align*}
	Therefore,
	\begin{align*}
		H_{n}(\mu_{k}(a,b,c)) & =\prod_{j=1}^{n}\left(\frac{u^{2}v}{x}\cdot q^{j-1}\cdot \left(\frac{uv q^{j-1}}{x}-1\right)\cdot (1-q^{j})(1-u q^{j-1})\right)^{n+1-j}\\
		& =\left(\frac{u^{2}v}{x^{2}}\right)^{\binom{n+1}{2}}q^{\binom{n+1}{3}}\prod_{j=1}^{n}(uv q^{j-1}-x)^{n+1-j}(u,q;q)_{n+1-j}.
	\end{align*}
	Finally, we apply \eqref{eq:xkDet} to \eqref{eq:alpha-mu}, and derive that
	\begin{align*}
		H_{n}(\widetilde{\alpha}_{k}^{u,v}(x)) = \left(\frac{x}{u}\right)^{n(n+1)} H_{n}(\mu_{k}(a,b,c)),
	\end{align*}
	which yields \eqref{eq:FinalExample} by recalling the above evaluation of $H_{n}(\mu_{k}(a,b,c))$.
\end{proof}

\begin{remark}
	From \cite[p.~21, Eq.~(1.13)]{ITZ}, we know that for the \emph{$q$-generalized Catalan numbers}: 
	\begin{align*}
		\lambda_{k}:=\frac{(u q;q)_{k}}{(uv q^{2};q)_{k}},
	\end{align*}
	their Hankel determinants are given by 
	\begin{align*}
		H_{n}(\lambda_{k})=u^{\binom{n+1}{2}}q^{\frac{n(n+1)(2n+1)}{6}}\prod_{j=1}^{n+1}\frac{(q,u q,v q;q)_{n+1-j}}{(uv q^{n+2-j};q)_{n+1-j}(uv q^{2};q)_{2(n+1-j)}}.
	\end{align*}
	Letting $u\mapsto u/q$ and $v\mapsto 0$ gives 
	\begin{align*}
		H_{n}((u;q)_{k})=u^{\binom{n+1}{2}}q^{2\binom{n+1}{3}}\prod_{j=1}^{n+1}(q,u;q)_{n+1-j}.
	\end{align*}
	Therefore, by \eqref{eq:xkDet}, 
	\begin{align*}
		H_{n}\big(q^{\binom{k}{2}}\alpha_{k}^{u,v}\big)=H_{n}((u;q)_{k}v^{k})=u^{\binom{n+1}{2}}v^{n(n+1)}q^{2\binom{n+1}{3}}\prod_{\ell=1}^{n+1}(u,q;q)_{n+1-\ell}.
	\end{align*}
	Hence, \eqref{eq:FinalExample} indicates that 
	\begin{align*}
		\frac{H_{n}(\widetilde{\alpha}_{k}^{u,v}(x))}{H_{n}\big(q^{\binom{k}{2}}\alpha_{k}^{u,v}\big)}=\frac{\prod_{j=0}^{n}(uv q^{j}-x)^{n-j}}{(uv)^{\binom{n+1}{2}}q^{\binom{n+1}{3}}}.
	\end{align*}
\end{remark}


\subsection*{Acknowledgements}

This project was sponsored by the Kunshan Municipal Government research funding.

\end{document}